\documentclass[12pt]{amsart}
\usepackage{amssymb,amsmath,txfonts,amsrefs, amscd}
\usepackage{array,multirow,bigstrut}
\usepackage{appendix,mathtools}
\usepackage[all]{xy}
\usepackage{xargs}  
\usepackage[pdftex,dvipsnames]{xcolor}  
\usepackage[colorinlistoftodos,prependcaption,textsize=tiny]{todonotes}
\usepackage{hyperref} 

\newcommandx{\unsure}[2][1=]{\todo[linecolor=red,backgroundcolor=red!25,bordercolor=red,#1]{#2}}

\newcommandx{\change}[2][1=]{\todo[linecolor=blue,backgroundcolor=blue!25,bordercolor=blue,#1]{#2}}

\newcommandx{\info}[2][1=]{\todo[linecolor=OliveGreen,backgroundcolor=OliveGreen!25,bordercolor=OliveGreen,#1]{#2}}

\newcommandx{\improvement}[2][1=]{\todo[linecolor=Plum,backgroundcolor=Plum!25,bordercolor=Plum,#1]{#2}}
\newcommandx{\thiswillnotshow}[2][1=]{\todo[disable,#1]{#2}}

\newtheorem{thm}{Theorem}[section]
\newtheorem{prop}[thm]{Proposition}
\newtheorem{lemma}[thm]{Lemma}
\newtheorem{remark}[thm]{Remark}

\newtheorem{defi}[thm]{Definition}

\DeclareMathOperator{\E}{\mathcal E}
\DeclareMathOperator{\Ima}{Im}
\numberwithin{equation}{section}

\newcommand{\ind}{\text {\rm{ind}\,}}
\newcommand{\cok}{\text {\rm{Coker}\,}}
     
\def\S{Steklov\,}

\def\HM{W^{k-\frac{1}{p},\,p}(\partial M^n)}

\def\WO{W^{k,\,p}( \Omega)}

\def\HO{W^{k-\frac{1}{p},\,p}(\partial \Omega)}
\def\ho{W^{k-1-\frac{1}{p},\,p}(\partial \Omega)}
\def\E{\rm{Emb}^m}

\def\R{\mathbb{R}}

\def\L{\Lambda}

\def\O{\Omega}

\def\PO{\partial \Omega}

\def\Ph{h^{\ast}}
\def\ph{h^{\ast-1}}
\def\n{\vec{n}}

\begin{document}

\title[]{Boundary Perturbations of Steklov Eigenvalues}

\author{Lihan Wang}
\email{lihan.wang@csulb.edu}
\address{Department of Mathematics and Statistics, California State University, Long Beach, 1250 Bellflower Blvd, Long Beach, CA 90840.}

\thanks{
The author was partially supported by NSF grant DMS-2316620 and DMS-2524167.}

\date{\today}


\begin{abstract}
We consider the dependence of non-zero Steklov eigenvalues on smooth perturbations of the domain boundary. We prove that these eigenvalues are generically simple under such boundary perturbations. This result complements our previous work on metric perturbations, thereby establishing generic simplicity Steklov eigenvalues under both fundamental geometric variations.

\smallskip

\end{abstract}

\maketitle
\pdfbookmark[0]{}{beg}

\section{Introduction}

Given a compact Riemannian manifold with smooth boundary, the Dirichlet-to-Neumann operator $\Lambda$ maps a function defined on the boundary to the normal derivative of its harmonic extension to the interior. Its eigenvalues are referred to as the Steklov eigenvalues. This spectral problem was first introduced by Steklov (\cite{S}) in 1902 with motivation from physics:
 harmonic extensions of the corresponding eigenfunctions describe a steady-state temperature on the domain for which the heat flux through the boundary is proportional to the temperature.

Steklov eigenvalue problem has since emerged as an active subject in mathematics. A seminal breakthrough by Fraser and Schoen revealed a profound link between the problem of maximizing Steklov eigenvalues and the rich theory of free boundary minimal surfaces in the unit Euclidean ball \cite{FS11,FS16,FS19}. Beyond pure geometry, the Steklov spectrum is fundamental to inverse problems. As the spectrum of the Dirichlet-to-Neumann operator, it provides crucial data for inverse conductivity and scattering problems, with important applications in non-destructive testing and imaging \cite{C, SU, GKLU}. We refer to \cite{GP} for a general survey and to \cite{CGGS} for more recent developments.

We are interested in the behavior of Steklov eigenvalues when the boundary is perturbed. For the simplicity, we focus on bounded domains with smooth boundary in $\R^n$. The topic of boundary perturbations and eigenvalues has a long
history and attracted many interests since the work of Rayleigh \cite{R1877} in 1877 (the first edition), Hadamard \cite{H1908} in 1908, Courant and Hilbert \cite{CH} in 1931. Since then, the study of boundary perturbations has remained an active field at the intersection of geometry and analysis. It also provides the foundation for applied fields such as shape optimization \cite{DZ2001}. 

One fundamental question in this topic is if the spectral simplicity  is a generic property with respect to perturbation of the boundary. In other words, if eigenvalues are simple on ``almost all" bounded domains, or a residual set of bounded domains. In her seminal work \cite{U}, Uhlenbeck investigate the case of the Laplace operator with the Dirichlet boundary condition and proved that the corresponding eigenvalues are simple for almost all domains. The same generic simplicity result was also obtained by Micheletti \cite{M72}. Later Henry \cite{H85,H86,H05} proved generic simplicity of eigenvalues for the Laplace operator with Robin or Neumann boundary conditions, and more general differential operators. All generic simplicity refers to non-zero eigenvalues.

In the Steklov case, the generic behavior of Steklov eigenvalues under boundary perturbations remains unexplored. The purpose of this paper is to establish the generic simplicity of Steklov eigenvalues. We prove that non-zero Steklov eigenvalues are simple for almost all bounded domains with smooth boundary. Precisely, our main theorem is stated as following:

\begin{thm}[Main Theorem]\label{MT}
Let $\O\in \R^n$ be a bounded domain with $C^m$ boundary for $m\geq 2$. Then there exists a residual set of $C^m$ embedding $h$ of $\O$ into $\R^n$ such that all non-zero Steklov eigenvalues on $h(\O)$ are simple.
\end{thm}

\begin{remark}This theorem also holds for compact manifolds with boundary that admit an embedding into $\R^n$. For a compact Riemannian manifold, such an embedding is guaranteed by the Nash embedding theorem.
\end{remark}

Our theorem completes a parallel investigation to our earlier work on metric perturbations \cite{W22}, establishing generic simplicity for Steklov eigenvalues under both fundamental geometric variations. This qualitative result also complements the local, quantitative analysis given by Hadamard-type formulas for Steklov eigenvalues such as \cite{DKL16}.

To establish the generic simplicity for \S eigenvalues, we follow the spirit of Uhlenbeck's work in \cite{U} by employing the infinite dimensional transversality theory. A central element in this approach is understanding how the Dirichlet-to-Neumann operator $\L$ varies under perturbations of the domain boundary. 

The Dirichlet-to-Neumann operator is a first-order, elliptic, self-adjoint pseudodifferential operator. For a typical differential operator, its variation under boundary deformation can be computed locally. In contrast, the variation of $\L f$ depends not only on the change in the boundary itself but also on how the harmonic extension of $f$ changes. To manage this complexity, we use a two-step variational method in \cite{W22}, deriving the required formula from the variation of harmonic extensions with respect to the boundary. We note that while Henry \cite{H05} obtained general variation formulas for differential operators under boundary perturbations, those results and methods do not directly apply to the nonlocal Dirichlet-to-Neumann operator for the reason stated.


A further challenge lies in verifying that zero is a regular value of certain evaluation maps. We construct such a map following techniques from \cite{U, W22}. In the context of metric perturbations, a density argument ensures that critical points of the evaluation map do not exist. For boundary perturbations of Steklov eigenvalues, however, critical points can occur. Using variational formulas, we characterize these critical points by an infinite-dimensional condition. Although we cannot eliminate them entirely, we can prove their occurrence is non-generic.

To achieve this, we introduce a second evaluation map, applying the framework developed by D. Henry in \cite{H05}, and invoke transversality theory. Due to the infinite-dimensional nature of the critical point condition, this new evaluation map has negative infinite Fredholm index. Consequently, standard transversality theorems are insufficient, and we must employ a generalized version proved by D. Henry (\cite{H05}). By combining the analysis from both evaluation maps, we conclude that Steklov eigenvalues are generically simple.



The article is organized as follows. In Section 2, we recall basic notions and facts concerning boundary perturbations. Section 3 is devoted to the variation of harmonic extensions and of the Dirichlet‑to‑Neumann operator with respect to boundary perturbations; we also restate the weak unique continuation property for Steklov eigenfunctions. In Section 4, we review the necessary background from transversality theory, including D. Henry's generalized transversality theorem. We then introduce two evaluation maps and employ them to prove the main result.

\section{Preliminaries}

We fix notation and collect the necessary preliminary material. Most of the definitions and results presented here are drawn from \cite{H05}.

Through out this paper, $\O$ stands for a bounded domian in $\R^n$ with $C^m$ boundary $\PO$ and $m,n\geq 2$. Here $C^m$ boundary means that there exists a $C^m$ function $\phi : \R^n\rightarrow \R$  which is at least in $\mathcal{C}_{unif}^1(\R^n, \R)$ such that $\O=\{x\in \R^n| \phi(x)>0\}$ and $\phi(x)=0$ implies that $|\nabla \phi| \geq 1$. 

We consider the set of $C^m$ embedding of $\O$ into $\R^n$ : 
 \begin{align}
\E(\O)&=\{h \in C^m(\O, \R^n)|\, \text{$h$ is injective and $\frac{1}{|\rm{det}( h^{\prime})|}$ is bounded in $\O$}\}
\end{align} which is an open subset of $C^m(\O, \R^n)$. The inclusion map of $\O$ into $\R^n$ will be denoted by $i_{\O}$.

For $h\in\E(\O) $, we denote its image by $\O_h=h(\O)$, the ``perturbed" domain. The perturbed boundary satisfies $\partial \O_h=h(\PO)$. We define the following pull-back map 
\begin{align}
\begin{aligned}
\Ph: \,&C^m(\O_h)\rightarrow C^m(\O)\\
&u\rightarrow u\circ h
\end{aligned}
\end{align} with inverse $\ph=(h^{-1})^{\ast}$.  And we use the notation $u_h=h^{\ast}u= u\circ h$ denotes the pull back of $u$.

The variation of $h$ will be denoted by $\dot{h}\in C^m(\O, \R^n)$. We can write the variation near the boundary as 
\begin{align*}
\dot{h}(x)&=\sigma(x) \n+\tau(x), x \in \PO
\end{align*} with the function $\sigma \in C^m(\PO)$ and the tangent vector field $\tau \bot \n$. To compute the variation with respect to $h$, we always consider a smooth curve $t \rightarrow h(, t)\in \E(\O)$ on some open interval $(a,b)$ and take the derivative with respect to $t$. In fact,
\begin{align*}
\frac{\partial h}{\partial t}(x, 0)&=\dot{h}(x)
\end{align*}where $x \in \O$ is fixed. In some calculations, however, we differentiate with respect to $t$ while keeping $y=h(x,t)$ fixed. To distinguish these two situations, we will explicitly indicate which variable is held fixed (e.g., $x=const.$,  $y=const.$)
 when taking derivatives.

Let $\n$ denote the unit outward normal along $\PO$. According to Corollary 1.6 in \cite{H05}, when $m\geq 2$, there exists a function $\phi\in C^m(\R^n, R)$ such that $\O=\{x| \phi(x)>0\}$ and $|\nabla \phi(x)|=1$ in a neighborhood of $\PO$. Moreover, such $\phi$ is unique on a neighborhood of $\PO$. In this case we can extend $\n= \nabla \phi$ on a neighborhood of $\PO$ and straightforward calculation yields the following properties. 
\begin{prop}\label{NP}
For a domain $\O \in \R^n$ with $C^m$ boundary $\PO$ and $m\geq 2$, its unit outward normal $\n$ satisfies:
\begin{enumerate}
\item the matrix $\left[\frac{\partial n_i}{\partial x_j}\right]_{i, j=1}^n$ is symmetric on $\PO$;
\item $\frac{\partial \n}{\partial \n}=0$ on $\PO$.
\end{enumerate}
\end{prop} 

Let $\Delta, \Delta_{\PO}$ denote the Laplacian on $\O$ and $\PO$ respectively. The second property in Proposition \ref{NP} implies the following relationship between them.
\begin{prop}\label{deco}
Let $\O$ be a bounded domain in $\R^n$ with $C^m$ boundary $\PO$ and $m\geq 2$. Then on $\PO$, there is
\begin{align*}
\Delta u
&=\Delta_{\PO} u+H u_n+u_{nn}.
\end{align*} Here $H=\nabla \cdot \n$ is the mean curvature of the boundary and $u_n= \nabla_{\n} u$. 
\end{prop} This proposition is a corollary of Theorem 1.13 in \cite{H05} or argument in Chapter 4 in \cite{Li12}.


For $h\in \E(\O)$, we can define $\n_{h}$ on a neighborhood of $\partial \O_h$ by
\begin{align*}
\n_{h(\O)}(y)=\n_{h(\O)}(h(x))&=\frac{(h^{-1}_{x})^{T}\n(x)}{\|(h^{-1}_{x})^{T}\n(x)\|}
\end{align*} which is the unit outward normal on $\partial \O_h$. Here $(h^{-1}_{x})^{T}$ is the inverse-transpose of the Jacobian matrix $h_x=[\frac{\partial h_i}{\partial x_j}]^n_{i,j=1}$ and $\|\cdot\|$ is the Euclidean norm. 

For $x$ near $\PO$, the image $y=h(x,t)$ is near $\PO(t)$. Hence we may compute the derivative $\frac{\partial \n_{h}}{\partial t}|_{y=const.}$. In fact, we have the following variation formula of $\n_h$ from Lemma 2.3 in \cite{H05}.
\begin{lemma}\label{var of n}
Let $\O$ be a bounded domain in $\R^n$ with $C^m$ boundary and $m\geq 2$. Suppose $h(\cdot, t)\in \E(\O)$ satisfies
\begin{align*}
\left\{
\begin{aligned}
&h(\cdot,0)=i_{\O},\\
&\frac{\partial}{\partial t}h(\cdot,,0)=\dot{h}.
\end{aligned}\right.
\end{align*} 
Here $\dot{h} \in C^m(\O, \R^n)$. Let $y=h(x,t)$ for $x\in \PO$. Then on $\PO$
\begin{align}\label{eqn: varN}
\frac{\partial \n_{h}}{\partial t}|_{y=const., t=0}=-\nabla_{\PO} \sigma
\end{align} with with $\sigma=\dot{h}\cdot \n$ along $\PO$.
\end{lemma}
\begin{remark}
In this variation lemma, as well as in the later variation lemmas, we only need to treat the special case $h(\cdot,0)=i_{\O}$, since this case is sufficient for the proof of the main results. 

In the general situation, let $h_0=h(\cdot, 0)$ and define \[\tilde{h}(x_0)=h(h_0^{-1}(x_0), t), \, x_0=h_0(x)\in h_0(\O).\] Then 
\[\tilde{h}(\cdot, 0)=i_{h_0(\O)}, \n_{h(t)}(x)=\n_{\tilde{h}(t)}(x_0), \frac{\partial \n_{h}}{\partial t}|_{y=const, t=0}(x)=\frac{\partial \n_{\tilde{h}}}{\partial t}|_{y=const., t=0}(x_0)\] for $x_0=h_0(x)$. Thus we can apply the lemma above to $\tilde{h}$ on $h_0(\O)$ to derive the variation formula of $h$. \end{remark}

\section{Variation formulas and the weak unique continuity principle}
Given $\O\subset \R^n$, the Dirichlet-Neumann operator is defined as 
\begin{align}\label{DN}
\begin{aligned}
\Lambda:& \HO\rightarrow \ho\\
& f\rightarrow u_n
\end{aligned}
\end{align} with integers $m-1 \geq k\geq1, p\geq 2$. Here $u$ is harmonic extension of $f$ on $\O$. That is, $u$ is the solution of the following boundary value problem:
\begin{equation}\label{he}
\left\{
\begin{aligned}
\Delta u&=0, &&\O\\
u&=f, &&\PO.
\end{aligned}\right.
\end{equation}  And $u_n$ denotes the outward normal derivative along $\PO$. As known, $\Lambda$ is elliptic and self-adjoint and has discrete spectrum. 

In fact, according to trace theorems (Chapter 2 in \cite{NJ}) and elliptic theory (Chapter 9 in \cite{GT}), for any $f \in \HO$, there exists the unique harmonic extension $u \in \WO$ of $f$  when $\O$ has $C^{k,1}$ boundary. And its normal derivative $u_n$ belongs to $\ho$.

For $h\in \E(\O)$, let $f_h=\ph(f)$ be the pull-back of $f$ on $\PO_h$. We introduce the following pulled-back operator on $\PO$:
\begin{align}\label{L_h}
\begin{aligned}
L_h=h^{\ast}\Lambda h^{\ast-1}: &\HO\rightarrow \ho, \\
&f\rightarrow h^{\ast}(\Lambda f_h).
\end{aligned}
\end{align}
Here $\Lambda$ acts on the reference domain $\PO_h$. Let $u^h$ denote the harmonic extension of $f_h$ on $\O_h$:
\begin{align}\label{eqn:u^h}
\left\{
\begin{aligned}
\Delta u^h&=0, &&\O_h\\
u^h&=f_h, &&\PO_h.
\end{aligned}\right.
\end{align} 
Then $L_h (f)=h^{\ast}(\Lambda f_h)=h^{\ast}(u^h_{n_h})$ on $\O$ with $u^h_{n_h}$ as the outward normal derivative of $u^h$ along $\PO_h$.

We first evaluate the variation of $u^h$ with respect to $h$ when $y\in \O_h$ is fixed. 
\begin{lemma}[Variation of $u^h$]\label{lemma:v1}
Suppose $h(\cdot, t)\in \E(\O)$ satisfies
\begin{align*}
\left\{
\begin{aligned}
&h(\cdot,0)=i_{\O},\\
&\frac{\partial}{\partial t}h(\cdot,,0)=\dot{h},
\end{aligned}\right.
\end{align*} 
with $\dot{h} \in C^m(\O, \R^n)$. 

For $f \in \HO$ and its pull-back $f_h=(h^{-1})^{\ast}f$, let $u$ and $u^h$ denote the corresponding harmonic extensions respectively. Keep $y=h(x,t)$ fixed in $\O_h$. Then  $\frac{\partial u^h}{\partial t}|_{y=const.}$ at $t=0$ satisfies 
\begin{align}\label{eqn:$D_tu^h$}
\left\{
\begin{aligned}
\Delta (\frac{\partial u^h}{\partial t}|_{y=const., t=0})&=0, &&\O\\
\frac{\partial u^h}{\partial t}|_{y=const.,t=0}&=-\nabla u\cdot \dot{h}, &&\PO\\
\end{aligned}\right.
\end{align}
\end{lemma}
\begin{proof}
By \eqref{eqn:u^h}, there is $\Delta u^h(y) =0$ for $y\in \O_h$. Take $\frac{\partial }{\partial t}$ on both sides with $y=h(x,t)$ fixed in $\O_h$. It follows that
\begin{align*}
&\frac{\partial }{\partial t}|_{y=const.}\left(\Delta u^h(y)\right) =0, \\
\Rightarrow& \Delta\left(\frac{\partial u^h}{\partial t}|_{y=const.} \right)=0.
\end{align*} 
The first equation in \eqref{eqn:$D_tu^h$} follows.

We rewrite the boundary condition of $u^h$ in \eqref{eqn:u^h} as
 \begin{align*}
 u^h(y)=u^h(h(x,t))&=f(x), \,  y=h(x,t)\in \PO_h.
\end{align*} With $x$ fixed in $\PO$, taking derivative in $t$ on both sides implies that
\begin{align*}
&\frac{\partial }{\partial t}|_{x= cont.}u^h(y)=0,\\
\Rightarrow &\frac{\partial }{\partial t}|_{y=const.}u^h(y)+(\nabla_yu^h)\cdot\frac{\partial y}{\partial t} =0,\\
\Rightarrow &\frac{\partial }{\partial t}|_{y=const.}u^h(y)=-\nabla u\cdot\frac{\partial y}{\partial t}.
\end{align*} The second equality follows from the chain rule. When $t=0$, there is $\frac{\partial y}{\partial t}=\frac{\partial h}{\partial t}=\dot{h}$ by the assumption. Thus it follows that 
\begin{align*}
&\frac{\partial u^h}{\partial t}|_{y=const., t=0}=-\nabla u\cdot \dot{h}.
\end{align*}The second equation in \eqref{eqn:$D_tu^h$} holds.

\end{proof}

Now we are ready to evaluate the variation of $L_h=h^{\ast}\Lambda (h^{-1})^{\ast}$ with respect to $h$ when $x\in \PO$ is fixed.

\begin{lemma}[Variation of $L_h$]\label{lemma:v2}
Under the assumptions of Lemma \ref{lemma:v1}, the variation of $L_h$ satisfies:
\begin{align}\label{eqn:v2 of Lh}
\frac{\partial L_h(f)}{\partial t}|_{x=const.,t=0}&=\nabla_{\PO}(\Lambda f)\cdot \dot{h}-\sigma (\Delta_{\PO}f+H(\Lambda f))-(\nabla_{\PO}\sigma) \cdot (\nabla_{\PO} f)-\Lambda(\nabla u\cdot \dot{h}).
\end{align}
 Here $H=\nabla \cdot \n$ is the mean curvature along $\PO$ and $\sigma= \dot{h}\cdot \n$. 

In addition, when $\dot{h}$ vanishes in the tangent direction, i.e., $\dot{h}=\sigma \n$ on $\PO$, there is
\begin{align}\label{eqn:v4 of Lh}
\frac{\partial L_h(f)}{\partial t}|_{x=const.,t=0}&=-\sigma \left(\Delta_{\PO}f+H(\Lambda f)\right)-(\nabla_{\PO}\sigma)\cdot (\nabla_{\PO} f)-\Lambda\left(\sigma (\Lambda f)\right).
\end{align}
\end{lemma}
\begin{proof}
Take $\frac{\partial}{\partial t}$ on both sides of \eqref{L_h} with $x\in \PO$ fixed. By the chain rule, it follows that
\begin{align}\label{eq:L_h1}
\begin{aligned}
\frac{\partial L_h(f)(x)}{\partial t}|_{x=const.}&=\frac{\partial}{\partial t}|_{x=const.}\left(h^{\ast}(\Lambda f_h)(x)\right)\\
&=\frac{\partial}{\partial t}|_{x=const.}\left(\Lambda f_h(y)\right), \, y=h(x,t)\\
&=\nabla_y(\Lambda f_h)\cdot \frac{\partial y}{\partial t}+\frac{\partial (\Lambda f_h)(y)}{\partial t}|_{y=const.}.
\end{aligned}
\end{align}

For the term $\nabla_y(\Lambda f_h)$, by the definition of $\Lambda$, there is
\begin{align*}
\nabla_y(\Lambda f_h)|_{t=0}&=\nabla_y( \n_h\cdot \nabla_y u^h)|_{t=0}=\nabla(\n \cdot \nabla u)\\
&=\nabla u_n=\nabla_{\PO} u_n+u_{nn} \n.
\end{align*} Here $\nabla$ and $\nabla_y$ denote the gradient on $\O$ and $\O_h$ respectively. At the same time, Proposition \ref{deco} implies that
\begin{align*}
u_{nn}&=\Delta u-\Delta_{\PO}u-Hu_n,\\
&=-\Delta_{\PO}u-Hu_n.
\end{align*} Thus 
\begin{align}\label{eq:L_h2}
\begin{aligned}
\nabla_y(\Lambda f_h)|_{t=0}&=\nabla_{\PO} u_n-(\Delta_{\PO}u+Hu_n)\n.
\end{aligned}
\end{align}

For the term $\frac{\partial}{\partial t}|_{y=const.}(\Lambda f_h)(y)$, by the definition of $\Lambda$, there is
\begin{align*}
\frac{\partial (\Lambda f_h)(y)}{\partial t}|_{y=const.}&=\frac{\partial}{\partial t}|_{y=const.}( \n_h\cdot \nabla_y u^h), \\
&=\frac{\partial \n_h}{\partial t}|_{y=const.}\cdot \nabla_y u^h+\n_h\cdot \frac{\partial}{\partial t}|_{y=const.}(\nabla_y u^h), \\
&=\frac{\partial \n_h}{\partial t}|_{y=const.}\cdot (\nabla_y u^h)+\n_h\cdot \nabla_y ( \frac{\partial u^h}{\partial t}|_{y=const.}).
\end{align*}
By \eqref{eqn: varN}, we have $\frac{\partial \n_h}{\partial t}|_{t=0}=-\nabla_{\PO} \sigma$. Plugging this into above formula yields that 
\begin{align}\label{eq:L_h3}
\begin{aligned}
\frac{\partial (\Lambda f_h)(y)}{\partial t}|_{y=const., t=0}&=-\nabla_{\PO} \sigma \cdot \nabla u+\n \cdot \nabla (\frac{\partial u^h}{\partial t}|_{y=const., t=0}),\\
&=-\nabla_{\PO} \sigma\cdot \nabla_{\PO} f+\nabla_n (\frac{\partial u^h}{\partial t}|_{y=const., t=0}).
\end{aligned}
\end{align} The second equality follows from $\nabla u=\nabla_{\PO} u+(\nabla_n u)\n=\nabla_{\PO} f+(\nabla_n u)\n$ on $\PO$. 

Plug \eqref{eq:L_h2} and \eqref{eq:L_h3} into \eqref{eq:L_h1}. Then we get
\begin{align}\label{eqn:v1 of Lh}
\begin{aligned}
\frac{\partial L_h(f)}{\partial t}|_{x=const., t=0}&=\nabla_{\PO}u_n\cdot \dot{h}-\sigma (\Delta_{\PO}f+Hu_n)-(\nabla_{\PO}\sigma) \cdot (\nabla_{\PO} f)\\
&+\nabla_n (\frac{\partial u^h}{\partial t}|_{y=const., t=0}).
\end{aligned}
\end{align} Here we use the fact that $\frac{\partial y}{\partial t}|_{t=0}=\frac{\partial h}{\partial t}|_{t=0}=\dot{h}$. At the same time, Lemma \ref{lemma:v1} implies that \[\nabla_n (\frac{\partial u^h}{\partial t}|_{y=const., t=0})=-\Lambda (\nabla u\cdot \dot{h}).\]  Plug this formula and $u_n=\Lambda f$ into \eqref{eqn:v1 of Lh}. Then \eqref{eqn:v2 of Lh} follows.
 
When $\dot{h}=\sigma \n$, there is \[\Lambda (\nabla u\cdot \dot{h})=\Lambda (\sigma u_n)=\Lambda (\sigma (\Lambda f)).\] Plug this equality into \eqref{eqn:v2 of Lh}. Then \eqref{eqn:v4 of Lh} follows.

\end{proof}

Finally, we introduce the weak unique–continuation principle for Steklov eigenfunctions, which is an essential tool in the proof of the main theorem. The following theorem is a consequence of the general version on Riemannian manifolds established in \cite{W22}, but it holds under weaker regularity assumptions. For completeness, we include a short proof.
\begin{thm}[Weak unique continuation principle]\label{wucp}
Let $\O$ be a bounded region in $\R^n$ with $C^m$ boundary for $m\geq 2$. Consider a non-constant function $f \in \HM$ with $m-1\geq k\geq 1$ and $p\geq 2$. Assume that $\Lambda f=\lambda f$. If $f=0$ on an open set of $\PO$, then $f$ vanish on $\PO$.
\end{thm}
\begin{proof}
Let $u$ be the harmonic extension of $f$. Then 
\begin{align*}\left\{ 
\begin{aligned}
\Delta u&=0,\, && \O\\ 
u&=f,\, && \PO\\
u_n&=\lambda f,\, && \PO.
\end{aligned}%
\right.
\end{align*} 
Let $\Gamma$ be an open subset of $\PO$ where $f=0$. Fix an arbitrary point $x\in \Gamma$. Choose a small ball $B(x,r)$ such that $B(x,r)\cap \PO\subset \Gamma$. 
We can extend $u$ to $B(x,r)$ as
\begin{align*}
\left\{ 
\begin{aligned}
\tilde{u}&=u,\, && B(x,r)\cap \O\ \\
\tilde{u}&=0,\, &&B(x,r) \setminus \O
\end{aligned}
\right.
\end{align*} such that $\tilde{u} \in W^{2,p}(B)$. Then we have
\begin{align*}
\Delta \tilde{u}&=0, \, 
\end{align*} on $B(x,r)$ and $\tilde{u}=0$ on the open set $B(x,r)\setminus \bar{\O}$. By the weak unique continuation principle of second order elliptic operators, $\tilde{u}$ must vanish in $B(x,r)$. As a harmonic function on $\O$, $u$ vanishing in the open set $B(x,r)\cap \Omega $ implies that $u$ vanishes everywhere in $\O$.\\
\end{proof}
\begin{remark}
See Theorem 1.1 in \cite{GL} for the weak unique continuation principle of second order elliptic operators. 
\end{remark}

\section{Main results}
In this section, we will establish the generic simplicity of Steklov eigenvalues under boundary perturbations using transversality theorems.  

Let us recall some basic definitions at first. 
\begin{defi}\label{rv}Consider a $C^1$ map $F: X\to Y$ between two manifolds. Say $x\in X$ is a regular point of $F$ if $D_xF : T_x (X) \to T_{F(x)}(Y)$ is onto. Say $y \in Y$ is a {\it regular value} if every point $x \in F^{-1}(y)$ is a regular point. 
\end{defi}

\begin{defi}
Say a subset of a topological space is {\it residual} if it is the countable intersection of open and dense sets.  Say a a subset of a topological space is {\it meager} if its complement is residual. 
\end{defi} 

\begin{defi}
Consider a continuous linear map $T$ between two Banach spaces. Say $T$ is {\it Fredholm} if its image is closed, and both of its kernel and cokernel are finite dimensional. Say $T$ is {\it left-Fredholm} if its image is closed and its kernel  is finite dimensional. The index is defined as $ \ind T=\dim \ker T-\dim \cok T$. 
\end{defi}


We will employ the following transversality theorem, which is adapted from \cite{H05} for convenience in our setting.

 \begin{thm}\label{T1}
 Suppose given positive integers $k,q$; Banach manifolds $X, Y, Z$ of class $C^k$ and $X,Y$ being separable; an open set $A \subset X\times Y$; a $C^k$ map of $f: A\rightarrow Z$; and a point $\zeta\in Z$. \\
 Assume for each $(x, y) \in f^{-1}(\zeta)$ that 
 \begin{enumerate}
 \item $\frac{\partial f}{\partial x}: T_x X\rightarrow T_{\zeta} Z$ is left Fredhlom and has  index $<k$;
 \item Either 
 \begin{enumerate}
 \item $Df_{(x,y)}=\left(\frac{\partial f}{\partial x}, \frac{\partial f}{\partial y}\right): T_x X\times T_y Y \rightarrow T_{\zeta} Z$ is surjective;
 \item or $\dim \left\{ \Ima Df_{(x,y)}\setminus \Ima \frac{\partial f}{\partial x}|_{(x,y)}\right\} \geq q+\rm{dim}\, \ker\frac{\partial f}{\partial x}|_{(x,y)}$.
 \end{enumerate}
 \end{enumerate}
 Let $A_y=\{x| (x, y)\in A\}$ and \[\dot{Y}_{crit}=\{y| \zeta \, \text{is a critical value of}\, f(\cdot, y): A_y \rightarrow Z.\] Then  $\dot{Y}_{crit}$ is  meager and closed in Y.
 
 Moreover, if the index of $\frac{\partial f}{\partial x} \leq -q<0$ on $ f^{-1}(\zeta)$, then 
  \[\dot{Y}_{crit}= \{y| f(x, y)=\zeta \, \text{for some x} \in A_y\}.\] 
 \end{thm}

\begin{remark}
Theorem \ref{T1} is an immediate corollary of Theorem 5.4 in \cite{H05}. The assumption that $X,Y$ are separable implies that the projection $(x,y) \rightarrow y: f^{-1}(\zeta)\rightarrow Y$ is proper (hence $\sigma-$proper), fulfilling the corresponding hypothesis.  Furthermore, condition (1) is equivalent to the semi-Fredholm condition with index $<k$ in in the original theorem.
\end{remark}

\subsection{First evaluation map}

Consider $f\in \HO$ with $u \in \WO$ as its harmonic extension. If $f$ is an eigenfunction corresponding to a non-zero \S eigenvalue $\lambda$, then $f$ is non-constant by \eqref{DN} and \eqref{he}. Furthermore, Green's formula yields 
\[\int_{\PO}f dA=\frac{1}{\lambda}\int_{\PO}u_n \cdot 1 dA=\frac{1}{\lambda}\int_{\O} \Delta u \,dV=0.\]
So we introduce the following space in order to apply Theorem \ref{T1}:
 \begin{align*}
 S_k^p(\PO)&=\left\{f \in \HO: \int_{\PO}f dA=0\right\}. 
 \end{align*}
And define the evaluation map as follows: 
\begin{align}\label{ev1}
\begin{aligned}
\phi:  S^p_k(\PO)\times R\times \E(\O)& \rightarrow \ho\\
(f, \lambda, h)&\rightarrow (L_h -\lambda) f=h^{\ast}\left((\Lambda-\lambda)(h^{-1})^{\ast}f\right).
\end{aligned}
\end{align}

The corresponding  tangent space can be characterized as
\begin{align*}
&S^p_k(\PO) \times \R \times C^m(\O, \R^n)|_{(f, \lambda, h)}&= \bigg\{
(\dot{f}, \dot{\lambda}, \dot{h}):\dot{f} \in \HO, \int_{\PO}\dot{f}  dA=0,\\
& \dot{\lambda} \in \R, \dot{h}\in C^m(\O, \R^n) \bigg\}.
\end{align*} 
And the differential of $\phi$ is given as follows: 
\begin{align}\label{DP}
D \phi_{(f, \lambda, h)}(\dot{f}, \dot{\lambda}, \dot{h})= L_h \dot{f}-\dot{\lambda} f+D_h (L_hf)(\dot{h}) .
\end{align} Here $D_h (L_h f)(\dot{h})$ denotes the variation of $L_h f=h^{\ast}\left((\Lambda-\lambda)(h^{-1})^{\ast}f\right)$ with respect to $h$ in the direction of $\dot{h}$, i.e., \[D_h (L_h f)(\dot{h})=\frac{\partial L_h(f)}{\partial t} |_{x=const., t=0}\] in Lemma \ref{lemma:v2}. 

For convenience, we will denote the differential of $\phi$ in the direction of $S^p_k(\PO) \times \mathbb{R}$ by $D_1$ and the differential in the direction $\E(\O)$ by $D_h$. That is,  
\begin{align}\label{Dif1}
D\phi=&(D_1 \phi, D_h \phi)\,\, \text{at $(f, \lambda, h)$}\\
\text{with}&\left\{
\begin{aligned}\label{Dif2}
&D_1 \phi(\dot{f}, \dot{\lambda})=D\phi(\dot{f}, \dot{\lambda},0)=(L_h-\lambda) \dot{f}-s \dot{\lambda}\\
&D_h \phi(\dot{h})=D\phi(0, 0, \dot{h})=(D_h L_h)(\dot{h}) f.
\end{aligned}
\right.
\end{align}

By the definition of $\phi$ \eqref{ev1}, we notice that $(f, \lambda, g) \in \phi^{-1}(0)$ is equivalent to $h^{\ast}\left((\Lambda-\lambda)(h^{-1})^{\ast}f\right)=0$ on $\PO$. That is, $(\Lambda-\lambda)(h^{-1})^{\ast}f=0$ on $\PO_h$. Thus the following lemma follows immediately.
\begin{lemma}\label{eva1}
$(f, \lambda, g) \in \phi^{-1}(0)$ if and only if $f_h=(h^{-1})^{\ast}f$ is an eigenfunction corresponding to the Steklov eigenvalue $\lambda\neq 0$ on $\O_h$.
\end{lemma}

\begin{lemma}\label{eva2}
For $h \in \E(\O)$, $0$ is a regular value of $\phi(, h)$ if and only if all non-zero Steklov eigenvalues on $\O_h$ are simple.
\end{lemma}
\begin{proof}
Fix some $h_0 \in \E(\O)$. Let $\O_0=h_0(\O)$. To show $0$ is a regular value of $\phi_0=\phi(, h_0)$, it is sufficient to prove that $D\phi_0$ is surjective at every $(f, \lambda)\in \phi_0^{-1}(0)$. 

Since $\phi_0(f, \lambda)=(L_{h_0}-\lambda)f$, there is
\begin{align*}
D\phi_0(\dot{f}, \dot{\lambda})&=(L_{h_0}-\lambda) \dot{f}-\dot{\lambda}f, \, (\dot{f}, \dot{\lambda})\in S^p_k(\PO)\times \R|_{(f, \lambda)}.
\end{align*} It follows that 
\begin{align*}
\Ima D\phi_0|_{(f, \lambda)}&=\Ima (L_{h_0}-\lambda)\oplus \rm{span} (f)\\
&=h_0^{\ast} \left\{\Ima (\Lambda-\lambda)|_{\O_0}\oplus \rm{span} ((h_0^{-1})^{\ast}f)\right\}.
\end{align*} Thus $D\phi_0$ is surjective at $(f, \lambda) \in \phi_0^{-1}(0)$ if and only if
\begin{align*}\Ima (\Lambda-\lambda)|_{\O_0}\oplus \rm{span} ((h_0^{-1})^{\ast}f&=W^{k-1-\frac{1}{p},p}(\PO_0).
\end{align*}

On the other hand, $\Lambda -\lambda$ is a self-adjoint elliptic operator on $\O_0$. Then there is\begin{align*}
W^{k-1-\frac{1}{p},p}(\PO_0)&=\rm{Im} (\Lambda- \lambda)|_{\O_0}\oplus \ker (\Lambda- \lambda)|_{\O_0}.
\end{align*} Therefore, $D\phi_0$ is surjective at $(f, \lambda) \in \phi_0^{-1}(0)$ if and only if
\begin{align*}
\rm{span} ((h_0^{-1})^{\ast}f)&=\ker (\Lambda- \lambda)|_{\O_0},
\end{align*} that is, $\lambda$ is a simple eigenvalue of $\Lambda$ on $\O_0$. The statement follows. 

\end{proof}

By Lemma \ref{eva2}, the main theorem will follow if we can show that  there exists a residual set of $h\in \E(\O)$ such that $\phi(\cdot, h)$ has $0$ as a regular value. To this end, we aim to apply Theorem \ref{T1} to the evaluation map $\phi$. Let $X=S^k_p\times \R, Y=C^m(\O, \R^n)$ which are separable Banach manifolds. Let $A=S^k_p\times \R\times \E(\O)$, an open set in $X\times Y$ and $\zeta=0$. Then 
\[\dot{Y}_{crit}=\{h\in \E(\O)\,| \,\text{0 is a critical value of $\phi(,h)$}\}.\] 
And its complement is $\{h\in \E(\O)\, |\, \text{0 is a regular value of $\phi(,h)$}\}$. 

We first verify the Fredholm condition for $\phi$.
\begin{lemma}\label{Fred} 
At $(f, \lambda, h)\in \phi^{-1}(0)$, the map
$D_1 \phi: S^k_p|_{(f, \lambda)} \times \R \rightarrow T$ is Fredhlom with index 0.
\end{lemma}
\begin{proof}
At $(f, \lambda, h)\in \phi^{-1}(0)$, there is
\begin{align*}
D_1 \phi|_{(f, \lambda, h)}(\dot{f}, \dot{\lambda}, \dot{h})&=(L_h-\lambda)\dot{f}-\dot{\lambda} f,\\
&=h^{\ast}\left((\Lambda-\lambda) \dot{f}_h- \dot{\lambda}f_h\right)
\end{align*} with $\dot{f}_h=h^{\ast-1}(\dot{f})$. Then 
\begin{align*}
\ker D_1\phi|_{(f, \lambda, h)}&=h^{\ast}\left(\ker (\Lambda-\lambda) \right)\oplus \{0\},\\
\Ima D_1 \phi|_{(f, \lambda, h)}&=h^{\ast}\left(\Ima (\Lambda-\lambda)\right)\oplus \rm{span} (f)
\end{align*} with $(\Lambda-\lambda)$ acting on $\PO_h$. We use the fact that $f_h\bot \Ima (\Lambda-\lambda)$ to get the second equality, since $f_h\in \ker (\Lambda-\lambda)$.

At the same time, the operator $(\Lambda-\lambda)|_{\PO_h}: W^{k-\frac{1}{p}}(\PO_h)\rightarrow W^{k-1-\frac{1}{p}}(\PO_h)$ is Fredholm with index $0$ for a fixed $\lambda$. The restriction on $S^p_k(\PO_h)$ decreases the index by 1. Therefore, $\dim \ker D_1 \phi|_{(f, \lambda, h)}=\dim \ker (\Lambda-\lambda)$ is finite dimensional. And $\Ima  D_1 \phi$ is closed since $\Ima (\Lambda-\lambda)$ and $ \rm{span} (f)$ are closed. Moreover, $\dim \cok D_1\phi|_{(f, \lambda, h)}=\dim \cok (\Lambda-\lambda)-1$. Therefore $D_1\phi$ is Fredholm with index $0$.
\end{proof}

We attempted to verify the condition 2(a) in Theorem \ref{T1}, namely that $0$ is a regular value of $\phi$, but failed. Instead, we obtain a special property for critical points of $\phi$.

\begin{prop}\label{idc}
If $(f, \lambda, h)$ is a critical point of $\phi$, then there exists non-trivial $\psi \in \HO$ such that its pull back $\psi_h=h^{\ast-1} (\psi)$ is a Steklov eigenfunction corresponding to $\lambda$ which is orthogonal to $f_h$ on $\PO_h$ and satisfies :
\begin{align}\label{eqn: density2}
&\nabla_{\PO} f_h\cdot \nabla_{\PO}\psi_h-(H +\lambda )\lambda f_h\psi_h=0, \, \PO_h.
\end{align} Here $H=\nabla \cdot \n$ is the mean curvature on $\PO_h$.
\end{prop}

\begin{proof}
Suppose that $(f, \lambda, h_0)$ is a critical point of $\phi$. Let $\O_0=h_0(\O)$. We can transfer the origin from $\O$ to $\O_0$ by defining the corresponding evaluation map $\tilde{\phi} (g, \lambda, h)=h^{\ast}\left((\Lambda-\lambda)h^{\ast-1}(g)\right)$ for $h\in \E(\O_0)$ and $g\in S^p_k(\PO_0)$. Then $\tilde{\phi} (g, \lambda, h)=\phi(h_0^{\ast-1}(g), \lambda, h\circ h_0)$. Moreover $(f, \lambda, h_0)$ is a critical point of $\phi$ if and only if $(h_0^{\ast-1}(f), \lambda, i_{\O_0})$ is a critical point of $\tilde{\phi}$. 

Therefore, without loss of generality, we assume that $(f, \lambda, i_{\O})$ is a critical point of $\phi$. We will make such change of origin without comments in later proofs.

First, we notice that $D\phi|_{(f, \lambda, i_{\O})}$ has a closed image. In fact, by \eqref{Dif1} and \eqref{Dif2}, we have that 
\begin{align*}
D\phi|_{(f, \lambda, i_{\O})}(\dot{f},0,0)&= (\Lambda -\lambda)\dot{f}.
\end{align*} That is, $\Ima(\Lambda -\lambda) \subset \Ima D\phi|_{(f, \lambda, i_{\O})}$. Since $\Ima(\Lambda -\lambda)$ is closed with finite co-dimension, $\Ima D\phi|_{(f, \lambda, i_{\O})}$ is also closed.

Thus $D\phi|_{(f, \lambda, i_{\O})}$ is not surjective if and only if the orthogonal complement of $\Ima D\phi_{(f, \lambda, i_{\O})}$ is non-trivial. That is, 
there exists a non-trivial $\psi \in \left(\rm{Im} D\phi|_{(f, \lambda, i_{\O})}\right)^{\bot}$.
By \eqref{DP}, for any $(\dot{f}, \dot{\lambda}, \dot{h})\in TS^k_p \times \R\times C^m(\O, \R^n)$, such $\psi$ satisfies
\begin{align}\label{ort}
0&=\int_{\PO}\psi\left((\Lambda -\lambda) \dot{f}-\dot{\lambda} f+(D_h L_h)(\dot{h}) f\right).
\end{align}

When $\dot{\lambda}=0, \dot{h}=0$ in \eqref{ort}, since $\Lambda$ is self-adjoint, it follows that 
\begin{align*}
0&=\int_{\PO}\psi\left((\Lambda -\lambda) \dot{f}\right)\\
&=\int_{\PO}\big((\Lambda -\lambda)\psi\big) \dot{f}
\end{align*} for any $\dot{f}\in C^m(\PO, \R^n)$. Thus, $(\Lambda -\lambda)\psi=0$, i.e., $\psi$ is a Steklov eigenfunction with corresponding to $\lambda$.

When $\dot{f}=0, \dot{h}=0$ in \eqref{ort}, it follows that 
\begin{align*}
0&=\int_{\PO}\psi\left(\dot{\lambda} f\right)=\dot{\lambda}\int_{\PO} \psi f
\end{align*} for any $\dot{\lambda}\in \R$. Thus $\int_{\PO} \psi f=0$, i.e., $\psi$ is orthogonal to $f$.

When $\dot{f}=0, \dot{\lambda}=0$ in \eqref{ort}, it follows that 
\begin{align}\label{density1}
0&=\int_{\O}\psi\left(D_h (L_h)(\dot{h})f\right)=\int_{\O}\psi\left(\frac{\partial L_h(f)}{\partial t}|_{x=const., t=0}\right)
\end{align} for any $\dot{h}\in C^m(\O, \R^n)$. 
To characterize the condition \ref{density1}, we choose $h(\cdot, t)\in \E(\O)$ such that \begin{align*}
\left\{
\begin{aligned}
&h(\cdot,0)=i_{\O},\\
&\frac{\partial}{\partial t}h(\cdot,0)=\dot{h}.
\end{aligned}\right.
\end{align*} 
and $\dot{h}|_{\PO}=\sigma \n$ for some $\sigma \in C^m(\PO)$. By \eqref{eqn:v4 of Lh}, there is 
\begin{align*}
\frac{\partial L_h(f)}{\partial t}|_{x=const., t=0}&=-\sigma \left(\Delta_{\PO}f+H(\Lambda f)\right)-(\nabla_{\PO}\sigma)\cdot (\nabla_{\PO} f)-\Lambda\left(\sigma (\Lambda f)\right),\\
&=-\sigma \left(\Delta_{\PO}f+H\lambda f\right)-(\nabla_{\PO}\sigma)\cdot (\nabla_{\PO} f)-\Lambda\left(\sigma \lambda f\right).
\end{align*}
Plug this formula into \ref{density1}. It follows that
\begin{align*}
0&=\int_{\PO} \psi \left(\frac{\partial L_h(f)}{\partial t}|_{x=const., t=0}\right),\\
&=\int_{\PO} \psi \left(-\sigma \left(\Delta_{\PO}f+H\lambda f\right)-(\nabla_{\PO}\sigma)\cdot (\nabla_{\PO} f)-\Lambda\left(\sigma \lambda f\right)\right),\\
&=\int_{\PO}(-\Delta_{\PO}f) (\psi\sigma) -\lambda \psi H f\sigma-\psi(\nabla_{\PO}\sigma)\cdot (\nabla_{\PO} f) -\lambda \psi \Lambda( f\sigma),\\
&=\int_{\PO}(\nabla_{\PO}f)\cdot (\nabla_{\PO} \psi)\sigma-\lambda \psi H f\sigma-\lambda \psi \Lambda( f\sigma)\\
&=\int_{\PO} \left(\nabla_{\PO} f\cdot \nabla_{\PO}\psi-\lambda f(H\psi+(\Lambda \psi))\right)\sigma.
\end{align*} The second last equality follows from Green's formula. The last equality use the fact that $\Lambda$ is self-adjoint. Since $\sigma$ is arbitrary and $(\Lambda-\lambda)\psi=0$, it follows
\begin{align*}
&\nabla_{\PO} f\cdot \nabla_{\PO}\psi-(H \psi+\Lambda  \psi)\lambda f=0\\
\Rightarrow& \nabla_{\PO} f\cdot \nabla_{\PO}\psi-(H +\lambda )\lambda f \psi=0.
\end{align*}  Then \eqref{eqn: density2} follows with the change of origin. 
\end{proof}

\begin{remark} The condition 2(b) in Theorem \ref{T1} does not hold for $\phi$ since $D_1 \phi$ has non-negative index at $(f, \lambda, h)\in \phi^{-1}(0)$. In fact, since \[\dim \cok D_1\phi \geq\dim \left\{\frac{\Ima (D\phi)}{\Ima(D_1 \phi)}\right\}\] at $(f, \lambda, h)$, there is 
\begin{align*}
\dim \left\{\frac{\Ima (D\phi)}{\Ima(D_1 \phi)}\right\}&\leq \dim \cok D_1\phi=\dim \ker D_1\phi-\ind D_1\phi \leq \dim \ker D_1\phi
\end{align*} at $(f, \lambda, h)$. Thus the condition 2(b) cannot hold for any positive integer $q$. 
\end{remark}

\subsection{Second evaluation map}
In this subsection, we show that condition \ref{eqn: density2} is non-generic. In particular, there exists an open and dense subset $E_1\subset \E(\O)$ such that the restriction of $\phi$ to $E_1$ has $0$ as a regular value. To prove this, we begin by introducing another evaluation map:

\begin{align}\label{ev2}
\begin{aligned}
\Phi:  &S^p_k(\PO)\times S^p_k(\PO)\times R\times \E(\O) \rightarrow \ho\times \ho\times \ho\\
&(f,\phi, \lambda, h)\rightarrow \left(\phi(f, \lambda, h), \phi(\psi, \lambda, h), \Psi(f, \psi, \lambda, h)\right)
\end{aligned}
\end{align} with
\begin{align*}
\left\{
\begin{aligned}
\phi(f, \lambda, h)&=h^{\ast}(\Lambda-\lambda) (h^{-1})^{\ast} f;\\
\phi(\psi, \lambda, h)&=h^{\ast}(\Lambda-\lambda) (h^{-1})^{\ast} \psi;\\
\Psi(f, \psi, \lambda, h)&=h^{\ast}\left(\nabla u^h\cdot \nabla w^h-(\nabla \cdot \n_h+2\lambda)\lambda  u^hw^h\right)|_{\PO}.
\end{aligned}\right.
\end{align*} Here $u^h, w^h$ denote the harmonic extension of $f_h, \psi_h$ respectively. 

The operator $\Psi$ comes from \eqref{eqn: density2}. In fact, when $\phi(f, \lambda, h)=\phi(\psi, \lambda, h)=0$, it follows that 
\begin{align}\label{Psidef}
\begin{aligned}
(h^{-1})^{\ast}\Psi(f, \psi, \lambda, h)&=\left(\nabla u^h\cdot \nabla w^h-(\nabla \cdot \n_h+2\lambda)\lambda  u^hw^h\right)|_{\PO_h}\\
&=\left((\nabla_{\PO_h}u^h+u^h_{n_h} \n_h)\cdot (\nabla_{\PO_h}w^h+w^h_{n_h} \n_h)-(\nabla \cdot \n_h+2\lambda)\lambda  u^hw^h\right)|_{\PO_h}\\
&=(\nabla_{\PO_h}u^h)\cdot (\nabla_{\PO_h}w^h)|_{\PO_h}+u^h_{n_h}w^h_{n_h}|{\PO_h} -(\nabla \cdot \n_h+2\lambda)\lambda  u^hw^h|_{\PO_h}\\
&=(\nabla_{\PO_h}f_h)\cdot (\nabla_{\PO_h}\psi_h)+\lambda^2 f_h\psi_h -(\nabla \cdot \n_h+2\lambda)\lambda  f_h \psi_h,\\
&=(\nabla_{\PO_h}f_h)\cdot (\nabla_{\PO_h}\psi_h) -(\nabla \cdot \n_h+\lambda)\lambda  f_h \psi_h.
\end{aligned}
\end{align} 

By \eqref{Psidef} and Lemma \ref{ev1}, we obtain the following characterization of the zeros of $\Phi$:
\begin{lemma}\label{zeros}
$\Phi(f, \psi, \lambda, h)=0$ if and only if 
\begin{enumerate}
\item both of $f_h, \psi_h$ are eigenfunction of $\Lambda$ on $\O_h$ corresponding to $\lambda\neq 0$.
\item $(\nabla_{\PO_h}f_h)\cdot (\nabla_{\PO_h}\psi_h) -(\nabla \cdot \n_h+\lambda)\lambda  f_h \psi_h=0$.
\end{enumerate}
\end{lemma}

We apply Theorem \ref{T1} to $\Phi$ to prove the following theorem.
\begin{thm}\label{MT3}
Let $\O\in \R^n$ be a bounded domain with $C^m$ boundary and $m\geq 2$. Then there exist an open and dense set $E_1\subset \E(\O)$ such that
\begin{align}\label{ev3}
\begin{aligned}
\phi:&S^k_p(\PO)\times \R\times E_1\rightarrow \ho\\
&(f, \lambda, h)\rightarrow h^{\ast}(\Lambda-\lambda) (h^{-1})^{\ast} f
\end{aligned}
\end{align} has zero as a regular value.
\end{thm}
\begin{proof}
To apply Theorem \ref{T1} to $\Phi$, let $X=S^p_k(\PO)\times S^p_k(\PO)\times \R, Y=C^m(\O, \R^n)$, both of which are separable, and $A= S^p_k(\PO)\times S^p_k(\PO)\times \R\times\E(\O)$. Let $D_1$ denote the differential in the direction of $S^p_k(\partial\Omega)\times S^p_k(\partial\Omega)\times\mathbb{R}$. 
We also use $D_1$ for the differential of $\phi$ in the direction of $S^p_k(\partial\Omega)\times\mathbb{R}$; the meaning will be clear from context.

Fix $(f,\psi,\lambda,h_0)\in \Phi^{-1}(0)$. Without loss of generality, we may assume that $(f,\psi,\lambda,i_{\O})\in \Phi^{-1}(0)$, since we can always translate the origin from $\Omega$ to $h_0(\Omega)$.

We first show that $D_1\Phi$ is left-Fredholm at each $(f, \psi, \lambda, i_{\O})\in \Phi^{-1}(0)$. For the kernel, there is
\begin{align*}
\ker D_1\Phi|_{(f, \psi, \lambda, i_{\O})}
&=\{(\dot{f}, \dot{\psi}, \dot{\lambda})| D_1\phi(\dot{f}, \dot{\lambda})=0, D_1\phi(\dot{\psi}, \dot{\lambda})=0, D_1 \Psi(\dot{f}, \dot{\psi},\dot{\lambda})=0\}.
\end{align*}
Lemma \ref{Fred} implies that $\dim \ker D_1\phi$ is finite, denoted by $l$. Then above equality implies that $\dim \ker D_1\Phi \leq l^2$.

For the image, there is \[\Ima D_1\Phi|_{(f, \psi, \lambda,  i_{\O})}=\Ima D_1\phi|_{(f, \lambda,  i_{\O})} \oplus\Ima D_1|_{( \psi, \lambda,  i_{\O})}\phi\oplus\Ima D_1\Psi|_{(f, \psi, \lambda,  i_{\O})}.\] Since $\Ima D_1\phi$ is closed by Lemma \ref{Fred}, it suffices to show that $\Ima D_1\Psi$ is closed to conclude that $\Ima D_1\Phi$ is closed. 

Direct calculation yields that
\begin{align}\label{D1}
\begin{aligned}
D_1\Psi(\dot{f}, \dot{\psi}, \dot{\lambda})|_{(f, \psi, \lambda, i_{\O})}&=\left[\nabla w\cdot\nabla- (\nabla\cdot \n+2 \lambda)\lambda w \right]\dot{f}^h|_{\PO}\\
&+\left[\nabla u\cdot\nabla- (\nabla\cdot \n+2 \lambda)\lambda u \right]\dot{\psi}^h|_{\PO}-4\lambda \dot{\lambda} f\psi
\end{aligned}
\end{align} with $u, w, \dot{f}^h, \dot{\psi}^h$ as harmonic extensions of $f, \psi, \dot{f}, \dot{\psi}$ respectively. Thus, $D_1\Psi$ is a linear combination of maps of two types.  
The first type is the map
\[
\dot{\lambda} \;\rightarrow\; -4 \lambda \, \dot{\lambda} \, f\psi,
\] 
which is closed, being simply multiplication of $\dot{\lambda}$ by the fixed function $-4\lambda f\psi$.  

The second type is the map
\[
\dot{f} \;\rightarrow\; \bigl[ \nabla w \cdot \nabla - (\nabla\cdot \n + 2\lambda)\lambda w \bigr] \, \dot{f}^h \big|_{\partial\Omega},
\] 
which is the composition of three operations: the harmonic extension $\dot{f}\mapsto \dot{f}^h$, the gradient, and multiplication by fixed functions.  
Both the gradient and multiplication operators are closed, and the harmonic extension is closed as a continuous map between Banach spaces.  

Therefore, $\Ima D_1\Psi$ is closed. And $D_1\Phi$ is left Fredholm at at $(f, \psi, \lambda, i_{\O})\in \Phi^{-1}(0)$.

To verify the index condition and condition (2) in Theorem \ref{T1}, we make the following claim: at $(f, \psi, \lambda, i_{\O})\in \Phi^{-1}(0)$, there is
\begin{align*}
\dim \left( \frac{\operatorname{Im}(D\Phi)}{\operatorname{Im}(D_1 \Phi)} \right) = \infty.
\end{align*}
Then the condition 2(b) holds. This claim will be proved as Proposition \ref{infinite dim}.  

Since 
\[
\dim \operatorname{coker} D_1 \Phi \;\geq\; \dim \left( \frac{\operatorname{Im}(D\Phi)}{\operatorname{Im}(D_1 \Phi)} \right),
\] 
it follows that the index of $D_1 \Phi$ is $-\infty$ at $(f, \psi, \lambda, i_{\O})\in \Phi^{-1}(0)$. 

Thus, by Theorem \ref{T1}, the subset 
\begin{align*}
\{h \in \E(\O)|\, \text{$\Phi(\cdot, h)$ has $0$ as a critical value}\}
\end{align*} is meager and closed. More over, it is equal to the subset
\begin{align*}
\{h \in \E(\O)|\, \text{ $\Phi(f, \psi, \lambda, h)=0$ for some $(f, \psi, \lambda)\in S^p_k(\PO)\times S^p_k(\PO)\times \R$}\}
\end{align*}whose complement is denoted by $E_1$. Hence $E_1$ is residual and open in $\E(\O)$.

We claim that the restriction of $\phi$ to $E_1$ has $0$ as a regular value:
\begin{align*}
\phi:\; & S^k_p(\partial\Omega) \times \mathbb{R} \times E_1 \;\longrightarrow\; \ho,\\
& (f, \lambda, h) \;\longmapsto\; h^{\ast}(\Lambda - \lambda) (h^{-1})^{\ast} f.
\end{align*}Suppose not and let $(f, \lambda, h) \in \phi^{-1}(0)$ be a critical point of $\phi$ with $h \in E_1$.  
Then, by Proposition \ref{idc} and Lemma \ref{zeros}, there exists $\psi \in S^k_p(\partial\Omega)$ such that $(f, \psi, \lambda, h) \in \Phi^{-1}(0)$.  
This contradicts the fact that $h \in E_1$.

\end{proof}

\begin{remark}
We recall that the harmonic extension is continuous. Indeed, if $u$ denotes the harmonic extension of $f$, it is the solution of the Dirichlet problem:
\begin{align*}
u(x)&=\int_{\PO} f(s) \nabla_{\n} G(x, s) ds, \, x\in \O.
\end{align*} with $G(x,s)$ as the Green's function satisfying Dirichlet boundary condition. 
\end{remark}

\begin{prop}\label{infinite dim}
Assume $(f, \psi, \lambda, i_{\O})\in \Phi^{-1}(0)$. Then 
\begin{align}\label{key}
\dim \left\{\frac{\Ima (D\Phi)}{\Ima(D_1 \Phi)}\right\}&=\infty
\end{align} at $(f, \psi, \lambda, i_{\O})\in \Phi^{-1}(0)$.
\end{prop}
To prepare for the proof of Proposition \ref{infinite dim}, we first derive the variation formula of $\Psi$ with respect to $h$.
\begin{lemma}\label{vari of Psi}
Assume that $(f, \psi, \lambda, i_{\O})\in \Psi^{-1}(0)$. For any $\dot{h}\in C^m(\PO, \R^n)$, the variation of $\Psi$ with respect to $h$ in the direction of $\dot{h}$ satisfies at $(f, \psi, \lambda, i_{\O})$:
\begin{align*}
 D_h\Psi(\dot{h})&=\left[\nabla w\cdot \nabla_x-(\nabla\cdot \n+2\lambda)\lambda w\right](-\dot{h}\cdot\nabla u)\\
&+\left[\nabla u\cdot \nabla_x-(\nabla\cdot \n+2\lambda)\lambda u\right](-\dot{h}\cdot\nabla w)\\
&+\lambda uw \Delta_{\PO}\sigma+\nabla \Psi \cdot \dot{h}
\end{align*}on $\PO$. Here $u, w$ denote the harmonic extension of $f, \psi$ respectively. 
\end{lemma}
\begin{proof} 
For connivence, for a fixed $\lambda$, we introduce the operator $\eta (f, \psi)=\Psi(f, \psi, \lambda, i_{\O})$:
\begin{align*}
\eta: &\HO \times \HO \rightarrow \ho\\
&(f, \psi)\rightarrow \left(\nabla u\cdot \nabla w-(\nabla \cdot \n+2\lambda)\lambda uw\right)|_{\PO}
\end{align*} with $u, w$ as the harmonic extension of $f, \psi$ respectively.  We notice that $\eta$ is well-defined on any domain in $\R^n$.

We define its pull-back operator as follows:
\begin{align*}
\eta_h=h^{\ast}\eta (h^{-1})^{\ast}: &\HO \times \HO \rightarrow \ho\\
&(f, \psi)\rightarrow h^{\ast}\left(\nabla u^h\cdot \nabla w^h-(\nabla \cdot \n_h+2\lambda)\lambda u^hw^h\right)|_{\PO}
\end{align*}with $u^h, w^h$ as the harmonic extension of $f_h=(h^{-1})^{\ast}f, \psi_h=(h^{-1})^{\ast}\psi$ respectively on $\O_h$. 

For $\dot{h}\in C^m(\O, \R^n)$, choose $h(\cdot, t)\in \E(\O)$ such that
\begin{align*}
\left\{
\begin{aligned}
&h(\cdot,0)=i_{\O},\\
&\frac{\partial}{\partial t}h(\cdot,0)=\dot{h}.
\end{aligned}\right.
\end{align*} 
Then $D_h\Psi(\dot{h})=\frac{\partial \eta_h(f, \psi)}{\partial t}|_{x=const, t=0}$. By the chain rule, we have
\begin{align}\label{varPsi1}
\begin{aligned}
D_h\Psi(\dot{h})(x)&=\frac{\partial \eta_h(f, \psi)}{\partial t}|_{x=const, t=0}(x), \, x\in \PO\\ 
&=\frac{\partial \eta (f^h, \psi^h)}{\partial t}(h(x))|_{x=const, t=0},\\
 &\overset{t=0}{=}\frac{\partial \eta (f^h, \psi^h)}{\partial t}|_{y=const.}(y)+\nabla_y \eta(f^h, \psi^h)\cdot \frac{\partial y}{\partial t}, \,  \, y=h(x,t)\\
 &\overset{t=0}{=}\frac{\partial \eta (f^h, \psi^h)}{\partial t}|_{y=const.}(y)+\nabla_x \eta (f, \psi)\cdot \dot{h}(x), \\
  &\overset{t=0}{=}\frac{\partial \eta (f^h, \psi^h)}{\partial t}|_{y=const.}(y)+\nabla_x \Psi\cdot \dot{h}(x).
  \end{aligned}
 \end{align}
Here \[\eta (f^h, \psi^h)=\left(\nabla_y u^h\cdot \nabla_y w^h-(\nabla_y \cdot \n_h+2\lambda)\lambda u^hw^h\right)|_{\PO_h}\] with $u^h, w^h$ as harmonic extensions of $f_h=(h^{-1})^{\ast}f, \psi_h=(h^{-1})^{\ast}\psi$ on $\O_h$ respectively. Thus the derivative $\frac{\partial \eta (f^h, \psi^h)}{\partial t}|_{y=const.}(y)$ satisfies
 \begin{align}\label{varPsi2}
 \begin{aligned}
\frac{\partial \eta (f^h, \psi^h)}{\partial t}|_{y=const.}(y)&\overset{t=0}{=}\left[\nabla_y w^h\cdot \nabla_y-(\nabla_y\cdot \n_h+2\lambda)\lambda w^h\right](\nabla_x u^h \cdot \frac{\partial x}{\partial t})\\
&+\left[\nabla_y u^h\cdot \nabla_y-(\nabla_y\cdot \n_h+2\lambda)\lambda u^h\right](\nabla_x w^h \cdot \frac{\partial x}{\partial t})\\
&-\nabla_y \cdot (\frac{\partial \n_h}{\partial t}|_{y=const.})\lambda u^hw^h,\\
&=\left[\nabla w\cdot \nabla_x-(\nabla\cdot \n+2\lambda)\lambda w\right](-\dot{h}\cdot\nabla u)\\
&+\left[\nabla u\cdot \nabla_x-(\nabla\cdot \n+2\lambda)\lambda u\right](-\dot{h}\cdot\nabla w)\\
&-\nabla\cdot  (\frac{\partial \n_h}{\partial t}|_{y=const.,t=0})\lambda uw.
\end{aligned}
\end{align} The second equality follows from that $\frac{\partial x}{\partial t}|_{t=0}=\frac{\partial h^{-1}}{\partial t}|_{t=0}= -h_x^{-1}h_t|_{t=0}=-\dot{h}$.

By Lemma \ref{var of n}, we have that $\frac{\partial \n_h}{\partial t}|_{y=const.,t=0}=-\nabla_{\PO} \sigma$ with $\sigma= \dot{h}\cdot \n$. Then \[\nabla\cdot  (\frac{\partial \n_h}{\partial t}|_{y=const.,t=0})=\nabla\cdot (-\nabla_{\PO} \sigma)=-\Delta_{\PO}\sigma.\] Plug this formula into \eqref{varPsi2} above. It follows that 
\begin{align*}
\frac{\partial \eta (f^h, \psi^h)}{\partial t}|_{y=const.}(y)&\overset{t=0}{=}\left[\nabla w\cdot \nabla_x-(\nabla\cdot \n+2\lambda)\lambda w\right](-\dot{h}\cdot\nabla u)\\
&+\left[\nabla u\cdot \nabla_x-(\nabla\cdot \n+2\lambda)\lambda u\right](-\dot{h}\cdot\nabla w)\\
&+\lambda uw \Delta_{\PO}\sigma.
\end{align*} Plugging this formula into \eqref{varPsi1} yields the desired result.

\end{proof}

Now we are ready to prove Proposition \ref{infinite dim}.
\begin{proof}[Proof of Proposition \ref{infinite dim}]
Assume \eqref{key} does not hold at $(f, \psi, \lambda, i_{\O})\in \Phi^{-1}(0)$. Let $\dim \left\{\frac{\Ima (D\Phi)}{\Ima(D_1 \Phi)}\right\}=J$ with a basis $\{\theta_1, \cdots, \theta_{J}\}$.  And $\theta_j=(\theta^1_j, \theta^2_j, \theta^3_j)$ for each $j$. Then, for any $(\dot{f}, \dot{\psi}, \dot{\lambda}, \dot{h})$, there exists $(v, w, s)\in S^k_p(\PO)\times S^k_p(\PO)\times \R$ and $(c_1, \cdots, c_{J}) \in \R^J$ such that $D\Phi(\dot{f}, \dot{\psi}, \dot{\lambda}, \dot{h})=D\Phi(v, w, s, 0)+\sum^{J}_{j=1}c_j \theta_j$, i.e. 
\begin{align*}
&D\Phi(\dot{f}-v, \dot{\psi}-w, \dot{\lambda}-s, \dot{h})=\sum^{J}_{j=1}c_j \theta_j.
\end{align*} Redefine $(\dot{f}, \dot{\psi}, \dot{\lambda}):=(\dot{f}-v, \dot{\psi}-w, \dot{\lambda}-s)$. Therefore, for any $\dot{h}$, there exists some $(\dot{f}, \dot{\psi}, \dot{\lambda})$ and constants $(c_1, \cdots, c_{J})$ such that 
\begin{align*}
D\Phi(\dot{f}, \dot{\psi}, \dot{\lambda}, \dot{h})&=\sum^{J}_{j=1}c_j \theta_j.
\end{align*} That is
\begin{align}
D\phi(\dot{f}, \dot{\lambda}, \dot{h})&=\sum^{J}_{j=1}c_j \theta^1_j;  \label{dh1}\\
D\phi(\dot{\psi}, \dot{\lambda}, \dot{h})&=\sum^{J}_{j=1}c_j \theta^2_j;\label{dh2}\\
\Psi(\dot{f}, \dot{\psi},  \dot{\lambda}, \dot{h})&=\sum^{J}_{j=1}c_j \theta^3_j.\label{dh3}
\end{align} 

We will first solve $\dot{f}, \dot{\lambda}, \dot{\psi}$ in terms of $\dot{h}$ by \eqref{dh1} and \eqref{dh2}. Then substitution of these formulas into \eqref{dh3} will imply an operator of $\dot{h}$ with finite rank. At the same time, the explicit calculation of $D\Psi$ will show such operator cannot be with finite rank. The contradiction implies that the assumption does not hold and \eqref{key} holds.

Step 1:  Solve $\dot{f}, \dot{\psi}, \dot{\lambda}$ in terms of $\dot{h}$ by \eqref{dh1} and \eqref{dh2}.

By \eqref{Dif1} and \eqref{Dif2}, there is 
\begin{align*}
D\phi(\dot{f}, \dot{\lambda}, \dot{h})&=D_1\phi(\dot{f}, \dot{\lambda})+D_h\phi( \dot{h})\\
&=(\Lambda-\lambda)\dot{f}-\dot{\lambda} f+D_hL_h(\dot{h})f.
\end{align*}
 Plugging this into \eqref{dh1} yields that
\begin{align}\label{dh11}
\begin{aligned}
&(\Lambda-\lambda)\dot{f}-\dot{\lambda} f+D_hL_h(\dot{h})f=\sum^{J}_{j=1}c_j \theta^1_j\\
\Rightarrow& (\Lambda-\lambda)\dot{f}=\sum^{J}_{j=1}c_j \theta_j^1+\dot{\lambda} f-D_hL_h(\dot{h})f.
\end{aligned}
\end{align}
Since $\Lambda$ is is self-adjoint and elliptic, its inverse operator is well-defined:
\[(\Lambda-\lambda)^{-1}: \Ima(\Lambda-\lambda)\rightarrow (\ker(\Lambda-\lambda))^{\bot}.\]  Let $\{\psi_j\}_{j=1}^l$ be the basis of $\ker(\Lambda-\lambda)$. Then \eqref{dh11} implies that there exist some constants $\{b_j\}_{j=1}^l$ such that
\begin{align}
 \dot{f}&=\sum^l_{j=1}b_j\psi_j+(\Lambda-\lambda)^{-1}\left(\sum^{J}_{j=1}c_j \theta_j^1+\dot{\lambda} f-D_hL_h(\dot{h})f)\right)\label{dh111}\\
\dot{\lambda}&=\frac{\int_{\PO}f\left(\sum^{J}_{j=1}c_j \theta_j^1-D_hL_h(\dot{h})f\right)}{\int_{\PO}f^2}.\label{dh112}
\end{align} Together \eqref{dh111} and \eqref{dh112} imply that $\dot{f}$ and $\dot{\lambda}$ are functions of $\dot{h}$.

From \eqref{dh2}, we can get similar formula for $\dot{\psi}$:
\begin{align}\label{dh21}
\dot{\psi}&=\sum^l_{j=1}d_j\psi_j+(\Lambda-\lambda)^{-1}\left(\sum^{J}_{j=1}c_j \theta_j^2+\dot{\lambda} \psi-D_hL_h(\dot{h})f\right)
\end{align} for some constants $d_j$. Together \eqref{dh21} and \eqref{dh112} imply that $\dot{\psi}$ is a function of $\dot{h}$. 

Therefore, substitution of \eqref{dh111}, \eqref{dh112} and \eqref{dh21} into \eqref{dh3} implies that the operator
\begin{align}\label{dh31}
\Gamma: \dot{h}&\rightarrow D\Psi (\dot{f}, \dot{\psi}, \dot{\lambda}, \dot{h})
\end{align} is with the finite rank.

Step 2: The calculation of $D\Psi$. \\
At $(f, \psi, \lambda, i_{\O})\in \Psi^{-1}(0)$, there is
\begin{align}
D\Psi(\dot{f}, \dot{\psi}, \dot{\lambda}, \dot{h})&=D_1\Psi(\dot{f}, \dot{\psi}, \dot{\lambda})+D_h\Psi(\dot{h}).
\end{align} 
 Let $u, w, \dot{f}^h, \dot{\psi}^h$ be corresponding harmonic extensions of $f, \psi, \dot{f}, \dot{\psi}$ respectively.  By \eqref{D1}, we have 
\begin{align}\label{dh31}
\begin{aligned}
D_1\Psi(\dot{f}, \dot{\psi}, \dot{\lambda})&=\left[\nabla w\cdot\nabla- (\nabla\cdot \n+2 \lambda)\lambda w \right]\dot{f}^h|_{\PO}\\
&+\left[\nabla u\cdot\nabla- (\nabla\cdot \n+2 \lambda)\lambda u \right]\dot{\psi}^h|_{\PO}-4\lambda \dot{\lambda} f\psi.
\end{aligned}
\end{align} Since
\begin{align*}
\left(\nabla w\cdot \nabla\dot{f}^h\right)|_{\PO}&=\nabla_{\PO} w\cdot \nabla_{\PO}\dot{f}^h+(\nabla_n w)( \nabla_n\dot{f}^h)\\
&=\nabla_{\PO} \psi\cdot \nabla_{\PO}\dot{f}+\lambda \psi  \nabla_n\dot{f}^h\\
&=\nabla_{\PO} \psi\cdot \nabla_{\PO}\dot{f}+\lambda \psi \Lambda \dot{f},
\end{align*} 
it follows that
\begin{align*}
\left(\nabla w\cdot \nabla-(\nabla\cdot \n+2\lambda)\lambda w\right)\dot{f}^h|_{\PO}&=\nabla_{\PO} \psi\cdot \nabla_{\PO}\dot{f}+\lambda \psi \Lambda \dot{f}-(\nabla\cdot \n+2\lambda)\lambda \psi \dot{f}.
\end{align*} Similarly, there is
\begin{align*}
\left(\nabla u\cdot \nabla-(\nabla\cdot \n+2\lambda)\lambda u\right)\dot{\psi}^h|_{\PO}&=\nabla_{\PO} f\cdot \nabla_{\PO}\dot{\psi}+\lambda f \Lambda \dot{\psi}-(\nabla\cdot \n+2\lambda)\lambda f \dot{\psi}.
\end{align*}
Plugging these two formulas into \eqref{dh31} yields that
\begin{align}\label{dh33}
\begin{aligned}
D_1\Psi(\dot{f}, \dot{\psi}, \dot{\lambda})&=\nabla_{\PO} \psi\cdot \nabla_{\PO}\dot{f}+\lambda \psi \Lambda \dot{f}-(\nabla\cdot \n+2\lambda)\lambda \dot{f}\\
&+\nabla_{\PO} f\cdot \nabla_{\PO}\dot{\psi}+\lambda f \Lambda \dot{\psi}-(\nabla\cdot \n+2\lambda)\lambda \dot{\psi}-4\lambda \dot{\lambda} f\psi.
\end{aligned}
\end{align} 
This implies that $D_1\Psi$ depends on $\dot{f}, \dot{\psi}$ up to order $1$ and $\dot{\lambda}$ up to oder $0$. 

At the same time, \eqref{eqn:v2 of Lh} implies that for any $\Lambda f=\lambda f$ 
\begin{align*}
D_hL_h(\dot{h})f&=\lambda \nabla_{\PO}f \cdot \dot{h}-\sigma (\Delta_{\PO}f+H(\lambda f))-(\nabla_{\PO}\sigma) \cdot (\nabla_{\PO} f)-\Lambda(\nabla u\cdot \dot{h})
\end{align*} which depends on $\dot{h}$ up to oder $1$. Together with \eqref{dh111}, \eqref{dh112} and \eqref{dh21}, this fact implies that $\dot{f}, \dot{\psi}$ and $\dot{\lambda}$ are functions of $\dot{h}$ up to $0$ order. Therefore, \eqref{dh33} implies that $D_1\Psi$ is an operator of $\dot{h}$ up to order $1$.

By Lemma \ref{vari of Psi}, we have 
\begin{align}\label{dh34}
\begin{aligned}
 D_h\Psi(\dot{h})&=\left[\nabla w\cdot \nabla_x-(\nabla\cdot \n+2\lambda)\lambda w\right](-\dot{h}\cdot\nabla u)\\
&+\left[\nabla u\cdot \nabla_x-(\nabla\cdot \n+2\lambda)\lambda u\right](-\dot{h}\cdot\nabla w)\\
&+\lambda uw \Delta_{\PO}\sigma+\nabla \Psi \cdot \dot{h}.
\end{aligned}
\end{align} This is an operator of $\dot{h}$ up to order $2$.

By \eqref{dh33} and \eqref{dh34}, we have:
\begin{align*}
\Gamma: \dot{h}&\rightarrow D\Psi (\dot{f}, \dot{\psi}, \dot{\lambda}, \dot{h}) =\lambda f \psi  \Delta_{\PO}\sigma +\text{lower order terms}.
\end{align*}  Since $\sigma=\dot{h}\cdot \n$ ranges over $C^m(\PO)$, the operator $\Delta_{\PO}\sigma$ cannot be of finite rank. For $\Gamma$ to be of finite rank, it would be necessary that $f \psi =0$ on $\PO$. However, this is impossible by the weak unique continuation principle for Steklov eigenfunctions in Theorem \ref{wucp}. Therefore, $\Gamma$ cannot be of finite rank. This yields a contradiction, and hence \eqref{key} holds.

Indeed, if $f \psi =0$ on $\PO$, then one of $f$ or $\psi$ must vanish on an open subset of $\PO$, and hence vanish identically by Theorem \ref{wucp}. Since both $f$ and $\psi$ are non-constant Steklov eigenfunctions, this is impossible.

\end{proof}
 
\subsection{Proof of Main theorem}
\begin{proof}
By Theorem \ref{MT3}, there exist an open and dense set $E_1\subset \E(\O)$ such that
\begin{align*}
\begin{aligned}
\phi:&S^k_p(\PO)\times \R\times E_1\rightarrow \ho\\
&(f, \lambda, h)\rightarrow h^{\ast}(\Lambda-\lambda) (h^{-1})^{\ast} f
\end{aligned}
\end{align*} has $0$ as a regular value.
Let $X=S^k_p(\PO)\times \R, Y=C^m(\O, \R^n),A=S^k_p\times \R\times E_1$ and $\zeta=0$. At any $(f, \lambda, h)\in \phi^{-1}(0)$ with $h\in E_1$,  $D_1\phi$ is Fredholm with index $0$ by Lemma \ref{Fred} and $D\phi$ is surjective. Applying Theorem \ref{T1} to this $\phi$ yields that
\[\dot{Y}_{crit}=\{h \in E_1\,| \,\text{0 is a critical value of $\phi(,h)$}\}\] is meager. Then its complement \[E_2=\{h\in E_1\, |\, \text{0 is a regular value of $\phi(,h)$}\}\] is residual. By Lemma \ref{eva2}, this implies that all non-zero Steklov eigenvalues of $\O_h$ are simple for $h\in E_2$.

\end{proof}

\appendix

\end{document}